\documentclass[12pt]{article} 

\usepackage[margin=1in]{geometry}

\usepackage{amssymb,amsfonts,amsmath,latexsym,amsthm}
\usepackage[usenames,dvipsnames]{color}
\usepackage[]{graphicx}
\usepackage[space]{grffile}
\usepackage{mathrsfs}   
\usepackage[skip=0pt]{caption}
\usepackage{subcaption}
\usepackage{verbatim}
\usepackage{url}
\usepackage{bm}
\usepackage{dsfont}
\usepackage{extarrows}
\usepackage{multirow}
\usepackage{tikz}
\usetikzlibrary{fit}					

\usepackage[colorlinks,citecolor=black,linkcolor=black,urlcolor=black]{hyperref}
\usepackage[authoryear,round]{natbib}
\usepackage[capitalise,noabbrev]{cleveref}

\theoremstyle{plain}
\newtheorem{theorem}{Theorem}[section]
\newtheorem*{theorem*}{Theorem} 

\newtheorem{proposition}[theorem]{Proposition}
\newtheorem{condition}[theorem]{Condition}

\theoremstyle{definition}

\theoremstyle{remark}

\allowdisplaybreaks

\newcommand{\rTheta}{{\tilde{\Theta}}}
\newcommand{\V}{\mathcal{V}}

\newcommand{\R}{\mathbb{R}}

\renewcommand{\Pr}{\mathbb{P}}

\newcommand{\X}{\mathcal{X}}

\newcommand{\branch}[4]{
\left\{
	\begin{array}{ll}
		#1  & \mbox{if } #2 \\
		#3 & \mbox{if } #4
	\end{array}
\right.
}

\def\app#1#2{%
  \mathrel{%
    \setbox0=\hbox{$#1\sim$}%
    \setbox2=\hbox{%
      \rlap{\hbox{$#1\propto$}}%
      \lower1.3\ht0\box0%
    }%
    \raise0.25\ht2\box2%
  }%
}

\usepackage{upgreek}

\title{Consistency of mixture models with a prior on the number of components}
\author{Jeffrey W. Miller\\
Harvard University, Department of Biostatistics}

\begin{document}
\maketitle
\begin{abstract}
This article establishes general conditions for posterior consistency of Bayesian finite mixture models with a prior on the number of components.
That is, we provide sufficient conditions under which the posterior concentrates on neighborhoods of the true parameter values when the data are generated from a finite mixture over the assumed family of component distributions.
Specifically, we establish almost sure consistency for the number of components, the mixture weights, and the component parameters, up to a permutation of the component labels.
The approach taken here is based on Doob's theorem, which has the advantage of holding under extraordinarily general conditions,
and the disadvantage of only guaranteeing consistency at a set of parameter values that has probability one under the prior.
However, we show that in fact, for commonly used choices of prior, this yields consistency at Lebesgue-almost all parameter values ---
which is satisfactory for most practical purposes.
We aim to formulate the results in a way that maximizes clarity, generality, and ease of use.
\end{abstract}

\noindent{\small AMS 2000 MSC: Primary: 62G20; Secondary: 62F15.}

\section{Introduction}

Many theoretical advances have been made in establishing posterior consistency and contraction rates for density estimation when using
nonparametric mixture models (see \citealp{ghosal2017fundamentals} and the many references therein)
or finite mixture models with a prior on the number of components \citep{kruijer2010adaptive,shen2013adaptive}.
Elegant results have also been provided showing posterior consistency and contraction rates for estimation of the discrete mixing distribution \citep{nguyen2013convergence} when using either class of models, as well as consistency for the number of components \citep{guha2021posterior}.


Meanwhile, it has long been known that Doob's theorem \citep{Doob_1949} can be used to prove almost sure consistency for the number of components as well as the mixture weights and the component parameters, up to a permutation \citep{Nobile_1994}.
Interestingly, in contrast to the modern theory mentioned above, a Doob-type result can be extraordinarily general, holding under very minimal conditions.
Doob's theorem has been criticized for only guaranteeing consistency on a set of probability one under the prior, and thus, a poorly chosen prior can lead to a useless result \citep{roeder1997practical}; however, for many models, this is a straw man argument since a well-chosen prior can lead to a consistency guarantee at Lebesgue almost-all parameter values.

While the result of \citet{Nobile_1994} was prescient and general, it has some disadvantages.
First, \citet{Nobile_1994} assumes some conditions that are not needed, specifically,
(i) that there is a sigma-finite measure $\mu$ such that for all $v$, the component distribution $F_v$ has a density $f_v$ with respect to $\mu$,
(ii) that $v \mapsto f_v(x)$ is continuous for all $x$, and
(iii) employing a somewhat complicated algorithm for mapping parameters into an identifiable space.
Further, it is difficult to use \citet{Nobile_1994} as a reference since the exposition is quite technical and requires significant effort to unpack.

In this article, we present a Doob-type consistency result for mixtures, with the goal of maximizing clarity, generality, and ease of use.
Our result generalizes upon the work of \citet{Nobile_1994} in that we do not require conditions (i)--(iii) above.
We formulate the result directly in terms of the original parameter space (rather than a transformed space as done by \citealp{Nobile_1994}), reflecting the way these models are used in practice.
Further, we provide conditions under which consistency holds almost everywhere with respect to Lebesgue measure,
rather than just almost everywhere with respect to the prior as done by \citet{Nobile_1994}.

Compared to the modern theory, the limitation of a Doob-type result is that for any given true parameter value, the theorem cannot tell us whether it is in the measure zero set where consistency may fail.
Another important caveat is that the data are required to be generated from the assumed class of finite mixture models.
Most consistency results are based on an assumption of model correctness, and the result we present is no different in that respect.
However, unfortunately, the posterior on the number of components in a mixture model is especially sensitive to model misspecification \citep{miller2018robust,cai2021finite}, so any inferences about the number of components should be viewed with extreme skepticism.
On the other hand, \citet{miller2013simple,miller2014inconsistency} show that popular nonparametric mixture models (such as Dirichlet process mixtures) are not even consistent for the number of components when the component family is correctly specified --- and this lack of consistency is an even more fundamental concern than sensitivity to misspecification.
Thus, although finite mixture models are rarely---if ever---exactly correct, having a consistency guarantee at least provides an assurance that the methodology is coherent.

In practice, mixture models with a prior on the number of components often provide useful insights into heterogeneous data and, as the saying goes, ``all models are wrong but some are useful'' \citep{box1979robustness}.
Mixtures are extensively used in a wide range of applications, and modern algorithms facilitate posterior inference when placing a prior on the number of components; see \citet{miller2018mixture} and references therein.
Thus, it is important to characterize the theoretical properties of these models as generally as possible.

The article is organized as follows.
In \cref{section:model}, we describe the class of models under consideration and we introduce the conditions to be assumed. 
In \cref{section:results}, we state our main results,
and \cref{section:proofs} contains the proofs.

\section{Model}
\label{section:model}

Let $(F_v: v\in \V)$ be a family of probability measures on $\X$, where $\V\subseteq\R^D$ is measurable and
$\X$ is a Borel measurable subset of a complete separable metric space, equipped with the Borel sigma-algebra.
For all $d$, we give $\R^d$ the Euclidean topology and the resulting Borel sigma-algebra.
For $k \in \{1,2,\ldots\}$, define
$\Delta_k := \{w\in(0,1)^k:\sum_{i = 1}^k w_i = 1\}\subseteq\R^k$.
For $w \in \Delta_k$ and $v\in\V^k$, define a probability measure
\begin{align}
\label{equation:distribution}
P_{w,v} =\sum_{i = 1}^k w_i F_{v_i}
\end{align}
on $\X$.  Thus, $P_{w,v}$ is the mixture with weights $w_i$ and component parameters $v_i$.

Let $\pi$, $D_k$, and $G_k$ be probability measures on $\{1,2,\ldots\}$, $\Delta_k$, and $\V^k$, respectively.
Consider the following model:
\begin{align}
\label{equation:model}
\begin{split}
\text{(number of components)}~~~~ & K\sim \pi \\
\text{(mixture weights)}~~~~ & W \mid K = k \,\,\sim D_k \text{ where $W = (W_1,\ldots,W_k)$ } \\
\text{(component parameters)}~~~~ & V \mid K = k \,\,\sim G_k \text{ where $V = (V_1,\ldots,V_k)$ } \\
\text{(observed data)}~~~~ & X_1,\ldots,X_n \mid W,V \,\,\sim P_{W,V} \text{ i.i.d. }
\end{split}
\end{align}

We use uppercase letters to denote random variables, such as $K$, and lowercase to denote particular values, such as $k$.

\subsection{Conditions} 
\label{section:conditions}

\begin{condition}[Family of component distributions]
\label{condition:components} ~
\begin{enumerate}
\item\label{condition:measurability} 
For all measurable $A\subseteq\X$, the function $v\mapsto F_v(A)$ is measurable on $\V$.
\item\label{condition:identifiability} (Finite mixture identifiability) 
For all $k,k'\in\{1,2,\ldots\}$, $w\in\Delta_k$, $w'\in\Delta_{k'}$, $v\in\V^k$, and $v'\in\V^{k'}$,
if $P_{w,v} = P_{w',v'}$ then $\sum_{i=1}^k w_i \delta_{v_i} = \sum_{i=1}^{k'} w_i' \delta_{v_i'}$.
\end{enumerate}
\end{condition}

Here, $\delta_x$ denotes the unit point mass at $x$.
Roughly,
\cref{condition:components}(\ref{condition:measurability}) is that $(F_v : v\in\V)$ is a measurable family and
\cref{condition:components}(\ref{condition:identifiability}) is that the discrete mixing distribution $\sum_{i=1}^k w_i \delta_{v_i}$ is uniquely determined by $P_{w,v}$.
Let $S_k$ denote the set of permutations of $\{1,\ldots,k\}$.

\begin{condition}[Prior] Under the model in \cref{equation:model}, for all $k\in\{1,2,\ldots\}$,
\label{condition:prior} ~
\begin{enumerate}
\item\label{condition:pi} $\Pr(K=k)>0$,
\item\label{condition:D} for all $A\subseteq\Delta_k$ measurable, if $\Pr(W\in A\mid K=k) = 0$ then $\{w_{1:k-1} : w\in A\}$ has Lebesgue measure zero,
\item\label{condition:G} for all $A\subseteq \V^k$ measurable, if $\sum_{\sigma\in S_k} \Pr(V_\sigma\in A \mid K=k) = 0$ then $A$ has Lebesgue measure zero,
\item\label{condition:distinct} $\Pr(V_i = V_j \mid K=k) = 0$ for all $1\leq i < j \leq k$.
\end{enumerate}
\end{condition}

Here, $w_{1:k-1} = (w_1,\ldots,w_{k-1})$ and $V_\sigma = (V_{\sigma_1},\ldots,V_{\sigma_k})$.
Roughly, Conditions \ref{condition:prior}(1--3) are that the prior gives positive mass to all $k$ and all sets with nonzero Lebesgue measure, for some permutation of the component labels.
\cref{condition:prior}(\ref{condition:distinct}) is that the component parameters are distinct with prior probability $1$.
Note that we do not assume $W|k$ and $V|k$ have densities with respect to Lebesgue measure.

\subsection{Examples} 

The conditions in \cref{section:conditions} hold for many commonly used mixture models.
For $F_v$, the family of component distributions, there are many commonly used choices that satisfy \cref{condition:components}(\ref{condition:identifiability}), 
including the multivariate normal \citep{yakowitz1968identifiability} 
and, more generally, many elliptical families such as the multivariate $t$ distributions \citep{holzmann2006identifiability}.
Several discrete families such as the
Poisson, geometric, negative binomial, and many other power-series distributions also satisfy \cref{condition:components}(\ref{condition:identifiability}) \citep{sapatinas1995identifiability}.
In each of these cases, \cref{condition:components}(\ref{condition:measurability}) can be easily verified using \citet[Theorem 2.37]{folland2013real}.

For the prior on the mixture weights $W|k$, \cref{condition:prior}(\ref{condition:D}) is satisfied by choosing $W|k \sim \mathrm{Dirichlet}(\alpha_{k 1},\ldots,\alpha_{k k})$ for any $\alpha_{k 1},\ldots,\alpha_{k k}>0$, since this has a density with respect to $(k-1)$-dimensional Lebesgue measure $d w_1 \cdots d w_{k-1}$ and this density is strictly positive on $\Delta_k$.  More generally, for the same reason, \cref{condition:prior}(\ref{condition:D}) is satisfied if $W|k$ is defined as follows: let $Z_i \sim \mathrm{Beta}(a_{k i},b_{k i})$ independently for $i\in\{1,\ldots,k-1\}$ where $a_{k i},b_{k i}>0$, then set $W_i = Z_i \prod_{j=1}^{i-1} (1-Z_j)$ for $i\in\{1,\ldots,k-1\}$ and $W_k = 1-\sum_{i=1}^{k-1} W_i$; this is called the generalized Dirichlet distribution \citep{ishwaran2001gibbs,connor1969concepts}.

For the prior on the component parameters $V|k$, perhaps the most common situation is that $V_1,\ldots,V_k$ are i.i.d.\ from some distribution $G_0$; in this case, Conditions \ref{condition:prior}(\ref{condition:G}) and \ref{condition:prior}(\ref{condition:distinct}) are satisfied if $G_0$ has a density with respect to Lebesgue measure and this density is strictly positive on $\V$ except for a set of Lebesgue measure zero.
A more interesting example is the case of repulsive mixtures, which use a non-independent prior on component parameters to favor well-separated mixture components.
For instance, \citet{petralia2012repulsive} propose defining $V|k$ to have a density (with respect to Lebesgue measure) proportional to $h(v) \prod_{i=1}^k g_0(v_i)$ where $g_0$ is a probability density on $\V$ and $h:\V^k\to\R$ is either $h(v) = \prod_{1\leq i< j\leq k} \rho(\|v_i - v_j\|)$ or 
$h(v) = \min_{1\leq i< j\leq k} \rho(\|v_i - v_j\|)$,
where $\rho:[0,\infty)\to\R$ is a strictly increasing, bounded function with $\rho(0) = 0$.
Then Conditions \ref{condition:prior}(\ref{condition:G}) and \ref{condition:prior}(\ref{condition:distinct}) are satisfied as long as $g_0$ is strictly positive on $\V$ except for a set of Lebesgue measure zero.

\section{Main results}
\label{section:results}

We show that for any model as in \cref{equation:model} satisfying Conditions \ref{condition:components} and \ref{condition:prior},
the posterior is consistent for $k$, $w$, and $v$ up to a permutation of the component labels, except on a set of Lebesgue measure zero.
More generally, if only Conditions \ref{condition:components} and \ref{condition:prior}(\ref{condition:distinct}) are satisfied, then
the result holds except on a set of prior measure zero.

Define $\Theta_k := \Delta_k \times \V^k$ and $\Theta :=\bigcup_{k = 1}^\infty \Theta_k$,
noting that $\Theta_1,\Theta_2,\ldots$ are disjoint sets.
Thus, for any $\theta\in\Theta$, we have $\theta = (w,v)$ for some unique $w\in\Delta_k$, $v\in\V^k$, and $k\in\{1,2,\ldots\}$;
let $k(\theta)$ denote this value of $k$.
In terms of $\theta$, the data distribution is $P_\theta = P_{w,v}$, 
where $P_{w,v}$ is defined in \cref{equation:distribution}.

We define a metric on $\Theta$ as follows: for $\theta,\theta'\in\Theta$, let
\begin{align}
\label{equation:d}
d_\Theta(\theta,\theta') =
    \left\{
	\begin{array}{ll}
		\min\{\|\theta - \theta'\|,1\}  & \text{if } k(\theta) = k(\theta'), \\
		1 & \text{otherwise,}
	\end{array}
\right.
\end{align}
where $\|\cdot\|$ is the Euclidean norm on $\Delta_k\times \V^k\subseteq\R^{k + k D}$.
\cref{proposition:union,proposition:subsets} show that $d_\Theta$ is indeed a metric
and $\Theta$ is a Borel measurable subset of a complete separable metric space; we give $\Theta$ the resulting Borel sigma-algebra.
Recall that $S_k$ denotes the set of permutations of $\{1,\ldots,k\}$.
For $\sigma\in S_k$ and $\theta\in\Theta_k$, let $\theta[\sigma]$ denote the transformation of $\theta$ obtained by permuting the component labels, that is, if $\theta = (w,v)$ then
$\theta[\sigma] := (w_\sigma,v_\sigma)$ where $w_\sigma = (w_{\sigma_1},\ldots,w_{\sigma_k})$ and $v_\sigma = (v_{\sigma_1},\ldots,v_{\sigma_k})$.
For $\theta_0\in\Theta_k$ and $\varepsilon>0$, define 
\begin{align}
\label{equation:Btilde}
\tilde{B}(\theta_0,\varepsilon) = \bigcup_{\sigma\in S_k} \big\{\theta\in\Theta : d_\Theta(\theta,\theta_0[\sigma]) < \varepsilon\big\}.
\end{align}

Consider the model in \cref{equation:model}, and define the random variable $\uptheta := (W,V)$.

\begin{theorem}
\label{theorem:consistency}
Assume Conditions \ref{condition:components} and \ref{condition:prior}(\ref{condition:distinct}) hold.
There exists $\Theta_* \subseteq \Theta$ such that $\Pr(\uptheta\in\Theta_*) = 1$ and for all $\theta_0\in\Theta_*$, 
if $X_1,X_2,\ldots\sim P_{\theta_0}$ i.i.d.\ then for all $\varepsilon>0$,
\begin{align}
\label{equation:consistency}
\lim_{n\to\infty} \Pr(\uptheta\in\tilde{B}(\theta_0,\varepsilon) \mid X_1,\ldots,X_n) = 1 ~~~\mathrm{a.s.}[P_{\theta_0}]
\end{align}
and
\begin{align}
\label{equation:K-consistency}
\lim_{n\to\infty} \Pr(K = k(\theta_0) \mid X_1,\ldots,X_n) = 1 ~~~\mathrm{a.s.}[P_{\theta_0}].
\end{align}
\end{theorem}

Here, the conditional probabilities are under the assumed model in \cref{equation:model};
note that $\uptheta\mid X_1,\ldots,X_n$ has a regular conditional distribution by \citet[Theorems 1.4.12 and 4.1.6]{durrett2019probability}.
Now, define a measure $\lambda$ on $\Theta$ as follows.
Let $\lambda_{\V^k}$ denote Lebesgue measure on $\V^k$, and let $\lambda_{\Delta_k}$ denote the measure on $\Delta_k$ such that, for all $A \subseteq \Delta_k$ measurable, 
$\lambda_{\Delta_k}(A)$ equals the Lebesgue measure of $\{w_{1:k-1} : w \in A\}\subseteq\R^{k-1}$.
Define $\lambda(A) := \sum_{k=1}^\infty (\lambda_{\Delta_k}\times \lambda_{\V^k})(A \cap \Theta_k)$
for all measurable $A\subseteq \Theta$.
In essence, $\lambda$ can be thought of as Lebesgue measure on $\Theta$.

\vspace{0em}
\begin{theorem}
\label{theorem:lebesgue}
If Conditions \ref{condition:components} and \ref{condition:prior} hold, then the set $\Theta_*$ in \cref{theorem:consistency}
can be chosen such that $\lambda(\Theta\setminus\Theta_*) = 0$.
\end{theorem}
\vspace{0em}

In other words, for $\lambda$-almost all values of $\theta_0$ in $\Theta$, 
if $X_1,X_2,\ldots\sim P_{\theta_0}$ i.i.d.\ then for all $\varepsilon>0$,
\cref{equation:consistency,equation:K-consistency} hold $P_{\theta_0}$-almost surely.


\section{Proofs}
\label{section:proofs}

\begin{proof}[{\bf Proof of \cref{theorem:consistency}}]
The basic idea of the proof is to use Doob's theorem on posterior consistency \citep{Doob_1949,miller2018detailed}.
However, Doob's theorem cannot be directly applied since it requires identifiability, and 
while we assume identifiability of $\sum_{i=1}^k w_i \delta_{v_i}$ in \cref{condition:components}(\ref{condition:identifiability}), this does not imply identifiability of $(w,v)$ due to 
(a) invariance of $P_{w,v}$ with respect to permutation of the component labels and (b) the existence of points in $\Theta$ where $v_i = v_j$.
To handle this, we consider a certain restricted parameter space on which identifiability holds for $(w,v)$, we apply Doob's theorem to a collapsed model on this restricted space, and we then show that this implies the claimed result on all of $\Theta$. 


\textit{Identifiability constraints.}
We constrain the component parameters as follows to obtain identifiability of $(w,v)$. Putting the dictionary order (also known as lexicographic order) on elements of $\V\subseteq\R^D$,
define
$$ \V_k := \big\{(v_1,\ldots,v_k)\in \V^k: v_1\prec\cdots\prec v_k\big\}\subseteq \R^{k D}. $$
Here, $v_i\prec v_j$ denotes that $v_i$ precedes $v_j$ and $v_i \neq v_j$.
Define $\rTheta_k := \Delta_k\times\V_k$ and $\rTheta := \bigcup_{k=1}^\infty \rTheta_k$.
Then $\rTheta$ is a Borel measurable subset of a complete separable metric space
under the metric $d_\Theta$ as defined in \cref{equation:d};
this follows from \cref{proposition:union,proposition:subsets} 
by taking $\X_k = \R^{k + k D}$, $d_k(x,y) = \|x - y\|$ for $x,y\in\X_k$, and $A_k = \rTheta_k$ for $k\in\{1,2,\ldots\}$.

\textit{Collapsed model.}
For $\theta\in\Theta_k$, define $T(\theta) = \theta[\sigma]$ where $\sigma\in S_k$ is chosen such that
$\theta[\sigma]\in\rTheta_k$ if possible, and otherwise $\theta[\sigma] = \theta$.
Then $\Pr(T(\uptheta)\in\rTheta) = 1$
since the subset of $\V^k$ where two or more $v_i$'s coincide has prior probability zero, by \cref{condition:prior}(\ref{condition:distinct}).
Denoting $B[\sigma] = \{\theta[\sigma] : \theta\in B\}$,
note that by the definition of $T$, for all $B\subseteq\rTheta_k$,
\begin{align}
\label{equation:relationship}
T^{-1}(B) = \{\theta\in\Theta : T(\theta)\in B\} = \textstyle\bigcup_{\sigma\in S_k} B[\sigma].
\end{align}
Letting $\tilde{Q}$ denote the distribution of $T(\uptheta)$, restricted to $\rTheta$, we have
\begin{align}
\label{equation:collapsed}
\begin{split}   
& T(\uptheta) \sim \tilde{Q} \\
& X_1,\ldots,X_n\mid T(\uptheta) \sim P_{T(\uptheta)} ~\text{ i.i.d.}
\end{split}
\end{align}
by \citet[Theorem 10.2.1]{dudley2002real} since $P_\theta = P_{T(\theta)}$ and 
for all $A\subseteq \X^n$ and $B\subseteq\rTheta$ measurable,
$\Pr(X_{1:n}\in A,\, T(\uptheta)\in B) = \Pr(X_{1:n}\in A,\, \uptheta\in T^{-1}(B)) = \int_B P_\theta^{(n)}(A) d\tilde{Q}(\theta)$,
where $X_{1:n} = (X_1,\ldots,X_n)$;
measurability of $\theta \mapsto P_\theta^{(n)}(A)$ for $A\subseteq\X^n$ follows from measurability of $\theta\mapsto P_\theta(A)$ for $A\subseteq\X$ (shown below at \cref{equation:measurable}) 
along with \citet[Lemma 5.2]{miller2018detailed}.
We refer to \cref{equation:collapsed} as the collapsed model.

\textit{Applying Doob's theorem.}
We show that the collapsed model in \cref{equation:collapsed} satisfies the conditions of Doob's theorem \citep{miller2018detailed}.
First, we check identifiability.
Let $\theta,\theta'\in\rTheta$ such that $P_\theta = P_{\theta'}$. 
By \cref{condition:components}(\ref{condition:identifiability}), $\sum_{i=1}^{k} w_i \delta_{v_i} = \sum_{i=1}^{k'} w_i' \delta_{v_i'}$
where $\theta = (w,v)$, $\theta' = (w',v')$, $k = k(\theta)$, and $k' = k(\theta')$.
By the definition of $\rTheta$, $v_1,\ldots,v_k$ are all distinct, $v_1',\ldots,v_{k'}'$ are all distinct, $w_1,\ldots,w_k > 0$, and $w_1',\ldots,w_{k'}' > 0$. 
This implies that $k=k'$, $w = w'_\sigma$, and $v = v'_\sigma$ for some $\sigma\in S_k$.
Further, because $v_1\prec\cdots\prec v_k$ and $v_1'\prec\cdots\prec v_k'$ by the definition of $\rTheta$,
it must be the case that $\sigma$ is the identity permutation, thus, $w = w'$ and $v = v'$, that is, $\theta = \theta'$.
Therefore, $\theta = (w,v)$ is identifiable on the restricted space $\rTheta$.

Next, we check measurability.
Let $A\subseteq\X$ be measurable. Then for any $k\in\{1,2,\ldots\}$,
\begin{align}
\label{equation:measurable}
\theta\mapsto P_\theta(A)=\sum_{i = 1}^k w_i F_{v_i}(A)
\end{align}
is measurable as a function on $\Theta_k = \Delta_k\times\V^k$, since the projections $(w,v)\mapsto w_i$ and $(w,v)\mapsto v_i$ are measurable, and $v_i\mapsto F_{v_i}(A)$ is measurable on $\V$ by \cref{condition:components}(\ref{condition:measurability}). 
Therefore, $\theta\mapsto P_\theta(A)$ is measurable as a function on $\rTheta_k=\Delta_k\times\V_k\subseteq\Delta_k\times\V^k$. It follows that it is measurable as a function on $\rTheta$ (since the pre-image of a measurable subset of $\R$ is a union of measurable subsets of $\rTheta_1,\rTheta_2,\ldots$ respectively, and is thus measurable by \cref{proposition:subsets} below).

Thus, by Doob's theorem \citep{miller2018detailed}, there exists $\rTheta_*\subseteq\rTheta$ such that 
$\Pr(T(\uptheta)\in \rTheta_*) = 1$ and the collapsed model is consistent at all $T(\theta_0)\in\rTheta_*$;
that is, for any neighborhood $B\subseteq\rTheta$ of $T(\theta_0)$, we have $\Pr(T(\uptheta)\in B \mid X_{1:n}) \to 1$ a.s.[$P_{T(\theta_0)}$].
Define $\Theta_*$ to be the set of all points in $\Theta$ that can be obtained by permuting the mixture components of a point in $\rTheta_*$, 
that is, $\Theta_* := \bigcup_{k=1}^\infty \bigcup_{\sigma\in S_k} (\rTheta_*\cap\rTheta_k)[\sigma]$. Then by \cref{equation:relationship},
\begin{align*}
\Pr(\uptheta\in\Theta_*) = \Pr(T(\uptheta)\in\rTheta_*) = 1.
\end{align*}


\textit{Putting the pieces together.}
Let $\theta_0\in\Theta_*$ and define $k_0 = k(\theta_0)$. 
Let $X_1,X_2,\ldots\sim P_{\theta_0}$ i.i.d.,  
let $\varepsilon\in (0,1)$, and define
$B := \{\theta\in\rTheta : d_\Theta(\theta,T(\theta_0)) < \varepsilon\} \subseteq \rTheta_{k_0}$. 
Referring to \cref{equation:Btilde}, observe that 
$\cup_{\sigma\in S_{k_0}} B[\sigma] \subseteq \tilde{B}(\theta_0,\varepsilon)$. 
Hence, by \cref{equation:relationship},

\begin{align}
\label{equation:convergence}
\Pr(\uptheta\in\tilde{B}(\theta_0,\varepsilon) \mid X_{1:n})
\geq \Pr(\uptheta\in\cup_{\sigma\in S_{k_0}} B[\sigma] \mid X_{1:n}) 
= \Pr(T(\uptheta)\in B \mid X_{1:n}) 
\xrightarrow[n\to\infty]{\mathrm{a.s.}} 1
\end{align}
since $P_{\theta_0} = P_{T(\theta_0)}$ and the collapsed model is consistent at all $T(\theta_0)\in\rTheta_*$.
This proves \cref{equation:consistency}.
\cref{equation:K-consistency} follows directly from \cref{equation:convergence}, since $\varepsilon<1$ implies $\tilde{B}(\theta_0,\varepsilon) \subseteq \Theta_{k_0}$, and therefore,
\begin{align*}
\Pr(K = k_0 \mid X_{1:n}) 
= \Pr(\uptheta \in \Theta_{k_0} \mid X_{1:n}) 
\geq \Pr(\uptheta\in\tilde{B}(\theta_0,\varepsilon) \mid X_{1:n})
\xrightarrow[n\to\infty]{\mathrm{a.s.}} 1.
\end{align*}
\end{proof}

\begin{proof}[{\bf Proof of \cref{theorem:lebesgue}}]
Define $\Theta_*$ as in the proof of \cref{theorem:consistency}.
Since $\Pr(\uptheta\in\Theta_*) = 1$, 
$$ 0 = \Pr(\uptheta\in\Theta\setminus\Theta_*) = \sum_{k=1}^\infty \Pr(\uptheta\in\Theta_k\setminus\Theta_* \mid K=k)\, \Pr(K = k). $$
Since $\Pr(K=k)>0$ for all $k$ by \cref{condition:prior}(\ref{condition:pi}), 
$\Pr(\uptheta\in\Theta_k\setminus\Theta_* \mid K=k) = 0$ for all $k$.



For $\sigma\in S_k$, let $D_k^{\sigma}$ and $G_k^{\sigma}$ denote the distributions of $W_\sigma | k$ and $V_\sigma | k$, respectively, under the model.
Note that for all $\sigma\in S_k$, $(\Theta_k\setminus\Theta_*)[\sigma] = \Theta_k\setminus\Theta_*$. Thus,
\begin{align}\label{equation:product-measure}
(D_k^{\sigma} \times G_k^{\sigma})(\Theta_k\setminus\Theta_*) 
= (D_k \times G_k)(\Theta_k\setminus\Theta_*) = \Pr(\uptheta\in\Theta_k\setminus\Theta_* \mid K=k) = 0.
\end{align}
Note that $\lambda_{\Delta_k}$ is invariant under permutations $\sigma\in S_k$, since by \citet[Theorem 2.47]{folland2013real}, Lebesgue measure $d w_1 \cdots d w_{k-1}$ on $\{w_{1:k-1}\in(0,1)^{k-1} : \sum_{i=1}^{k-1} w_i < 1\}$ is invariant under transformations of the form $g(w_{1:k-1}) = (w_{\sigma_1},\ldots,w_{\sigma_{k-1}})$ where $w_k = 1 - \sum_{i=1}^{k-1} w_i$, because the Jacobian determinant is $\pm 1$.
Conditions \ref{condition:prior}(\ref{condition:D}) and \ref{condition:prior}(\ref{condition:G})
are that $\lambda_{\Delta_k} \ll D_k$ and $\lambda_{\V^k} \ll \sum_{\sigma\in S_k} G_k^{\sigma}$, respectively,
where $\ll$ denotes absolute continuity.
Thus, by \citet[Exercise 3.2.12]{folland2013real},
\begin{align}
\label{equation:absolute-continuity}
\lambda_{\Delta_k}\times \lambda_{\V^k} \ll \lambda_{\Delta_k} \times \sum_{\sigma\in S_k} G_k^{\sigma}
= \sum_{\sigma\in S_k} \lambda_{\Delta_k}^{\sigma} \times G_k^{\sigma}
\ll  \sum_{\sigma\in S_k} D_k^{\sigma} \times G_k^{\sigma}.
\end{align}
By \cref{equation:product-measure}, $(D_k^{\sigma} \times G_k^{\sigma})(\Theta_k\setminus\Theta_*) = 0$ for all $\sigma\in S_k$,
and thus, $(\lambda_{\Delta_k}\times \lambda_{\V^k})(\Theta_k\setminus\Theta_*) = 0$ by \cref{equation:absolute-continuity}.
Therefore,
$\lambda(\Theta\setminus\Theta_*) = \sum_{k=1}^\infty (\lambda_{\Delta_k}\times \lambda_{\V^k})(\Theta_k\setminus\Theta_*) = 0$.
\end{proof}

\section*{Acknowledgments}
Thanks to Matt Harrison for helpful comments on an early version of this manuscript.

\appendix

\section{Supporting results}
\label{section:details}

\begin{proposition}
\label{proposition:union}
If $\X_1,\X_2,\ldots$ is a sequence of disjoint, complete separable metric spaces with metrics $d_1,d_2,\ldots$ respectively, then $\X =\bigcup_{i = 1}^\infty \X_i$ is a complete separable metric space under the metric
$$ d(x,y) =\branch{\min\{d_i(x,y), 1\}}{x,y\in \X_i \text{ for some } i,}
                  {1}{x\in \X_i,\,y\in \X_j, \text{ and } i\neq j,} $$
and the topology induced by this metric coincides with the disjoint union topology.
\end{proposition}

The disjoint union topology is the smallest topology that contains all the open sets of all the $\X_i$'s. Equivalently, it is the topology consisting of all unions of the form $\bigcup_{i = 1}^\infty A_i$ where $A_i$ is open in $\X_i$ for $i\in\{1,2,\ldots\}$.

\begin{proof}
First, we show that $d$ is a metric on $\X$. It is easy to see that $d(x,y) = d(y,x)$, $d(x,y)\geq 0$, and $d(x,y) = 0 \iff x = y$. To prove the triangle inequality, let $x,y,z\in \X$ and suppose $x\in \X_i$, $y\in \X_j$, $z\in \X_k$. Using the fact that $\bar d(x,y) := \min\{d(x,y), 1\}$ is a metric \citep[Theorem 20.1]{munkres2000topology}, it is simple to check that $d(x,y)\leq d(x,z) + d(z,y)$ in each of the following cases:
(1) $i = j = k$, 
(2) $i = j\neq k$, and
(3) $i\neq j$.

Next, we show that $\X$ is complete under $d$.  Let $x_1,x_2,\ldots\in \X$ be a Cauchy sequence. Choose $N$ such that for all $n,m\geq N$, $d(x_n,x_m)\leq 1/2$. Suppose $i$ is the index such that $x_N\in \X_i$. Then $x_n\in \X_i$ for all $n\geq N$, and $d(x_n,x_m) = d_i(x_n,x_m)$ for all $n,m\geq N$.  Thus, $(x_N,x_{N +1},\ldots)$ is a Cauchy sequence in $\X_i$ under $d_i$, so it converges (under $d_i$) to some $x\in \X_i$ since $\X_i$ is complete. Hence, it also converges to $x$ under $d$. Therefore, $\X$ is complete.

Further, $\X$ is separable, since if $C_i\subseteq \X_i$ is a countable dense subset of $\X_i$ under $d_i$ then it is also dense in $\X_i$ under $d$, so $\bigcup_{i = 1}^\infty C_i$ is a countable dense subset of $\X$ under $d$.

Finally, $d$ induces the disjoint union topology on $\X$, since the collection of open balls
$$\left\{B_\varepsilon(x):\varepsilon\in(0,1),\,x\in \X_i,\,i= 1,2,\ldots\right\} $$
where $B_\varepsilon(x) =\{y\in \X: d(x,y)<\varepsilon\}$ is a base for both the disjoint union topology and the $d$-metric topology.
\end{proof}

\begin{proposition}
\label{proposition:subsets}
Suppose $\X_1,\X_2,\ldots$ and $\X$ are defined as in \cref{proposition:union}. If $A_1,A_2,\ldots$ are Borel measurable subsets of $\X_1,\X_2,\ldots$, respectively, then $\bigcup_{i = 1}^\infty A_i$ is a Borel measurable subset of $\X$.
\end{proposition}
\begin{proof}
For a topological space $Y$, let $\mathcal{T}_Y$ denote its topology and let $\mathcal{B}_Y=\sigma(\mathcal{T}_Y)$ denote its Borel sigma-algebra. Since $\mathcal{T}_{\X_i}\subseteq\mathcal{T}_\X$ (by the definition of the disjoint union topology) then $\mathcal{B}_{\X_i}\subseteq\mathcal{B}_\X$, and therefore $A_i\in\mathcal{B}_{\X_i}\subseteq\mathcal{B}_\X$ for all $i = 1,2,\ldots$. Hence, $\bigcup_{i = 1}^\infty A_i\in\mathcal{B}_\X$.
\end{proof}


\bibliography{references}
\bibliographystyle{abbrvnatcaplf}

\end{document}